\documentclass[a4paper,UKenglish]{lipics-v2016}
\usepackage{microtype}%if unwanted, comment out or use option "draft"

\usepackage{todonotes}%if unwanted, use option "disable"

\makeatletter
\newcommand\xleftrightarrow[2][]{%
  \ext@arrow 9999{\longleftrightarrowfill@}{#1}{#2}}
\newcommand\longleftrightarrowfill@{%
  \arrowfill@\leftarrow\relbar\rightarrow}
\makeatother

\title{Discrete Gradient Line Fields on Surfaces}
\author[1]{Thomas Lewiner}
\author[2]{Tiago Novello}
\author[3]{Jo\~{a}o Paix\~{a}o}
\author[4]{Carlos Tomei}

\affil[1]{Gamma, The Boston Consulting Group, Paris, France\\
  \texttt{thomas@lewiner.org}}
\affil[2]{Department of Mathematics, PUC-Rio, Rio de Janeiro, Brazil\\
  \texttt{tiago.novello@mat.puc-rio.br}}
\affil[3]{Department of Mathematics, UFRJ, Rio de Janeiro, Brazil\\
  \texttt{jpaixao@dcc.ufrj.br}}
\affil[4]{Department of Mathematics, PUC-Rio, Rio de Janeiro, Brazil\\
  \texttt{tomei@mat.puc-rio.br}}

\authorrunning{T. Lewiner, T. Novello, J. Paixão, and C. Tomei} %mandatory. First: Use abbreviated first/middle names. Second (only in severe cases): Use first author plus 'et. al.'

% mandatory: Please provide 1-5 keywords
\begin{document}

\maketitle

\begin{abstract}
{
A line field on a manifold is a smooth map which assigns a tangent line to all but a finite number of points of the manifold. As such, it can be seen as a generalization of vector fields. They model a number of geometric and physical properties, e.g. the principal curvature directions dynamics on surfaces or the stress flux in elasticity. 

We propose a discretization of a Morse-Smale line field on surfaces, extending Forman's construction for discrete vector fields. More general critical elements and their indices are defined from local matchings, for which Euler theorem and the characterization of homotopy type in terms of critical cells still hold.
} 
\end{abstract}
%=======================================================================

%=======================================================================
\section{Introduction}
%=======================================================================
{A line field on a manifold is a smooth map which assigns a tangent line to all but a finite number of points. Such fields model a number of physical properties, like velocity and temperature gradient in fluid flow, stress and momentum flux in elasticity.
Recently, line fields have earned eminence in the nematic fields ambiance \cite{lubensky98,stark99}. In computer graphics, line field are common tools for quadrangulation \cite{dong2006}, visualization of vector/line/symmetric tensor fields \cite{tricoche2002,delmarcelle94}, and remeshing \cite{alliez2003}. Curious line fields abound: some are illustrated in  Figure \ref{figure:linefield_nature}. Fingerprints are studied by the  pattern recognition community \cite{jinwei2004}. 

\begin{figure}[!h]
\centering
\includegraphics[width=13.5cm]{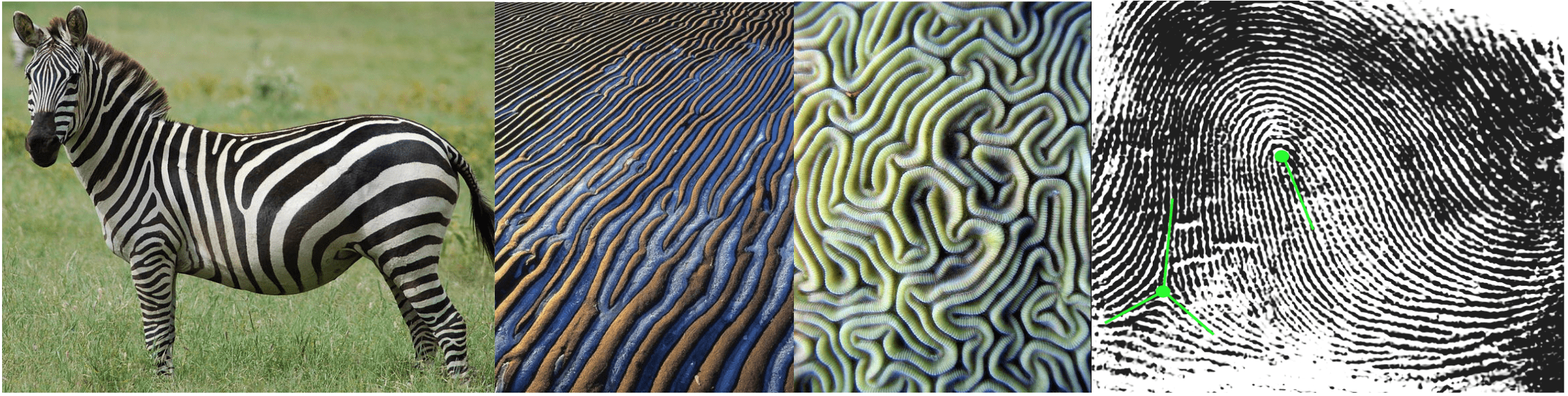}
	\caption{Lines fields: a zebra, sand, coral, fingerprints.}
	\label{figure:linefield_nature}
\end{figure}

There are natural numerical issues, which are common to vector and line fields: how does one compute critical points or special orbits joining them? The literature is extensive (e.g.
\cite{tricoche2001, tricoche2001_1,weinkauf2005, klein2007, chen2007, weinkauf2008}). We take instead a  point of view which is frequent within the dynamical systems community: given a field, we consider its structurally stable properties. For this, one prescribes a metric in the space of fields in a fixed surface $S$ (induced by the $C^r$ topology, for $r \ge 1$), and  an equivalence relation between fields --- topological equivalence --- given by the existence of a homeomorphism in $S$ taking orbits of one field to another (a fine, detailed description of the many objects involved is \cite{nikolaev97}). A field $X_0$ is structurally stable if it is equivalent to sufficiently close fields $X$, i.e., there is $\epsilon$ such that if $| X - X_0 | < \epsilon$ then $X$ and $X_0$ are equivalent.

Andronov and Pontryagin \cite{andronov1937} identified the now called Morse-Smale vector fields as being structurally stable, and Peixoto \cite{peixoto1962structural,peixoto1963structural,peixoto73} proved that they are the only such fields on compact, oriented surfaces. Bronshteyn and Nikolaev \cite{bronshteyn1998structurally,nikolaev98} defined Morse-Smale line fields and showed their structural stability. Such  fields  contain only finitely many  critical points and cycles. Line fields admit some stable critical points with no counterpart in the continuous case,  shown in Figure \ref{figure:critical_point}. Morse-Smale  vector and line fields are dense with respect to the metric: fields arising from a physical or computational situation are well approximated by a structurally stable  field.
\begin{figure}[!h]
\begin{center}
	\includegraphics[width=11cm]{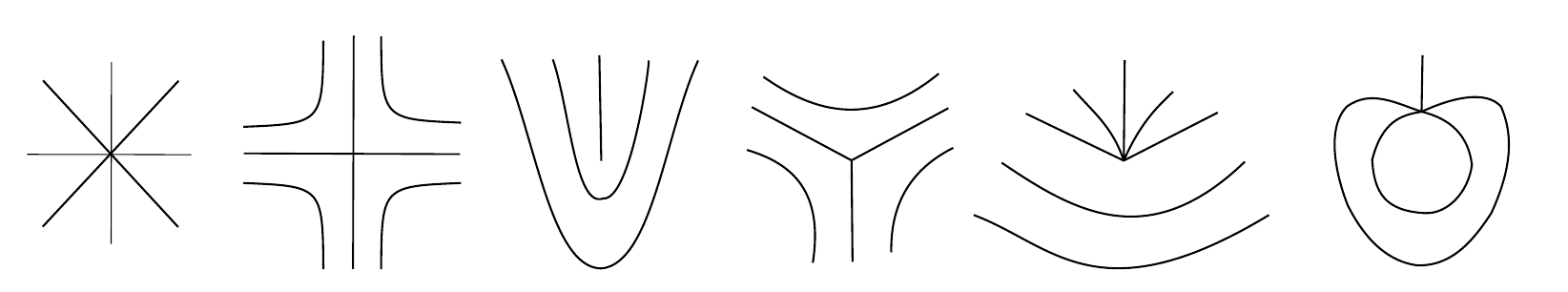}
	\caption{The structurally stable critical points: only the first two are present in vector fields.}
	\label{figure:critical_point}
\end{center}
\end{figure}

Peixoto also defined a fundamental combinatorial object, named the Peixoto orgraph in \cite{nikolaev97}, which in a sense contains all the stable aspects:  classes of Morse-Smale vector fields with respect to topological equivalence are in bijection with the possible Peixoto orgraphs. The analogous result for line fields was presented in \cite{nikolaev97}.

\begin{figure}[h]
\begin{center}
	\includegraphics[width=10cm]{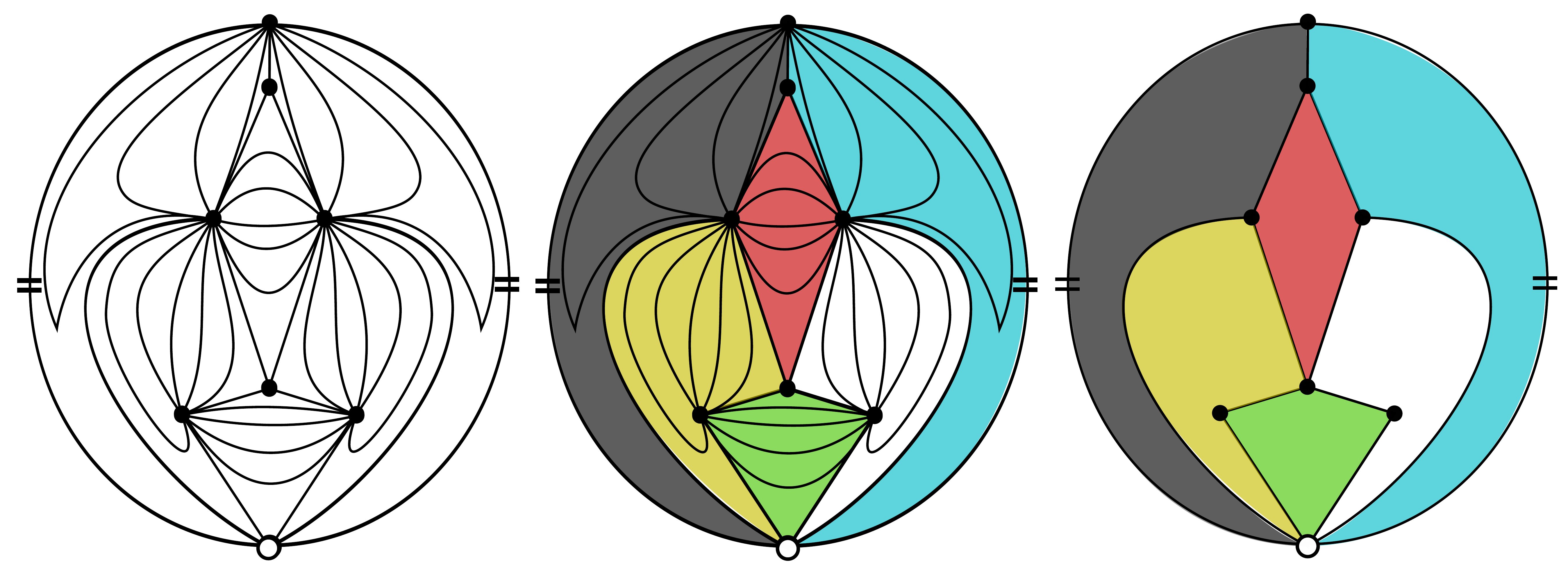}
	\caption{Extracting a Peixoto orgraph of the sphere from a  line field.}
	\label{figure:peixto_invariant}
\end{center}
\end{figure} 
		
Forman \cite{forman98, forman98_1} introduced discrete vector fields as a combinatorial counterpart to Morse-Smale vector fields. For oriented surfaces, such objects  give rise to (all possible) Peixoto orgraphs: in a sense,  discrete and continuous vector fields are  equivalent for stability purposes. Forman's theory has been used in computational geometry (\cite{lewiner2003, lewiner2005, paixao14}), yielding a robust approach to (discrete) Morse functions and the computation of Peixoto orgraphs \cite{jan11,jan11_1}. 

In this text we propose a discretization of a subclass of Morse-Smale line field on compact surfaces, extending Forman's construction. We only allow critical points given by the first four cases in Figure  \ref{figure:critical_point} and those  presented in Figure \ref{figure:critical_points_foliation}, of possible physical interest, despite of the fact that they are not structurally stable. 

\begin{figure}[h]
\begin{center}
	\includegraphics[width=6cm]{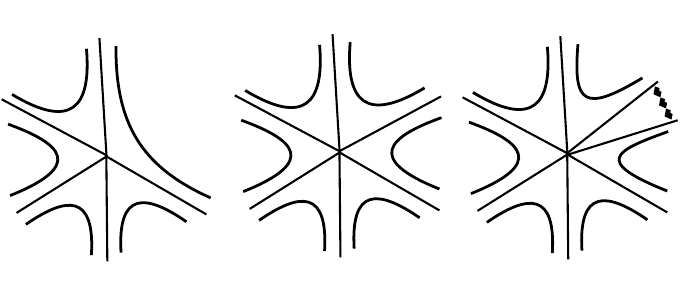}
	\caption{Non structurally stable critical points amenable to discretization}
	\label{figure:critical_points_foliation}
\end{center}
\end{figure} 

As for the continuous cases and Forman's theory, our construction  obtains an Euler-Poincaré theorem and a Peixoto orgraph which encapsulates the underlying topology. 
}
%=======================================================================

%=======================================================================
\section{Basic vocabulary}
%=======================================================================
{
We follow  \cite{gross1987topological}. An \textit{embedding} $i:G \to S$ of a graph $G$ into a connected, compact surface $S$ is a $1$-$1$ continuous map. Two embeddings $i_1$ and $i_2$ of $G$ in a surface $S$ are {\it equivalent} if there exists a homeomorphism $h: S\to S $ such that $h \circ i_1 = i_2$. A {\it 2-cell embedding} $i: G \to S$ is an embedding for which the components of the set $S\setminus i(G)$ are homeomorphic to open discs, the 2-cells. 
A 2-cell embedding $i:G \to S$ induces a {\it (cell) decomposition} $K=(V,E,F)$ of $S$, where $i(G)=(V,E)$ in the obvious way and $F$ are the 2-cells of $S\setminus i(G)$. Given a 2-cell $\Sigma$, the edges and vertices in the boundary of $\Sigma$ can be arranged as a closed walk $\sigma_1,\sigma_2,\ldots,\sigma_k$ along the edges of $G$, the \textit{boundary walk} of $\Sigma$. 

The {\it Hasse diagram} $H(K)$ of $K$ is a graph embedding whose vertices  consist of a unique point in each cell of $K$, and edges are  disjoint lines  connecting  points of adjacent cells whose dimension differ by one (see Figure \ref{figure:MorseMatching}). We abuse language slightly and think of the vertices of $H(K)$ as the cells of $K$. The edges of $H(K)$ split into those containing  vertices or faces of $K$, inducing  the subgraphs $H_V$ and $H_F$.
\begin{figure}[!h]
\begin{center}
	\includegraphics[width=13cm]{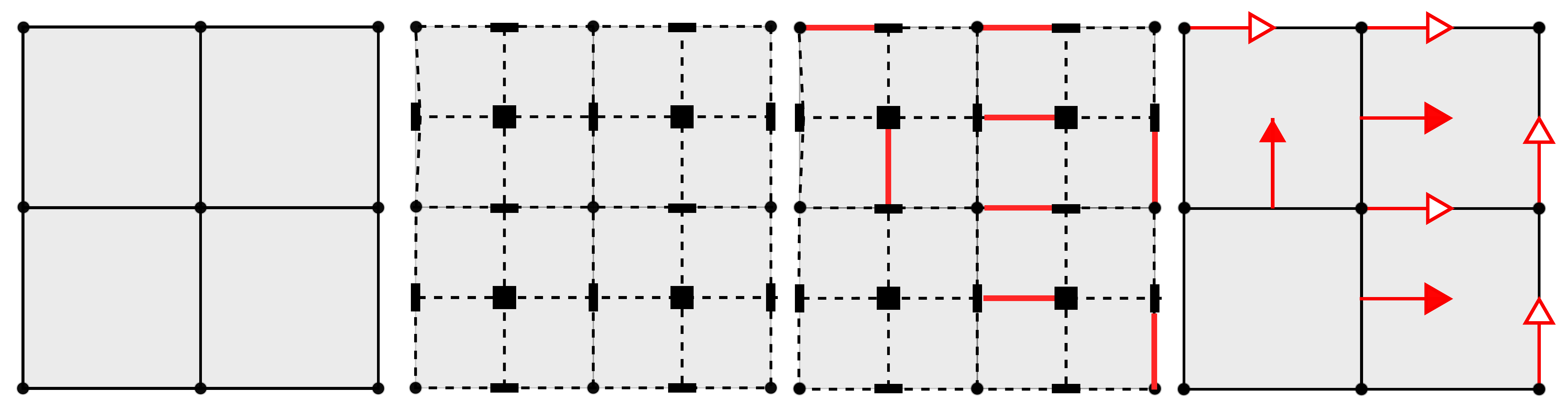}
	\caption{$K$, its Hasse diagram $H(K)$ and two representations of a Morse matching.}
	\label{figure:MorseMatching}
\end{center}
\end{figure}

A {\it matching in $H(K)$} is a collection $M$ of disjoint edges in the Hasse diagram $H(K)$. A {\it discrete vector field} is a pair $(K,M)$. Clearly $M$ is the disjoint union of the sets of edges $M_V$ and $M_F$ containing elements respectively in $V$ and  in $F$. Following \cite{forman98_1,chari2000}, a {\it Morse matching} is a matching in $H(K)$ for which neither $M_V$ or $M_F$ contains a set of alternate edges of a closed cycle of $H(K)$. A {\it discrete gradient vector field} is a pair $(K, M)$ for a Morse matching $M$. Such fields correspond to Forman's  discrete, acyclic Morse-Smale vector fields.

An unmatched cell in a Morse matching is a \textit{critical cell} and its \textit{index} is its dimension.  The discrete version on the Poincaré-Euler formula is due to Forman \cite{forman98}. 

\begin{theorem}[Euler-Poincaré]
\label{theorem:forman_euler_formula} Let $(K,M)$ be a discrete gradient vector field of a compact surface $S$, and $m_p$ be the number of unmatched $p$-cells. Then
 $$\chi(S)=\displaystyle \sum_{i=0,1,2}(-1)^i m_i\ .$$
\end{theorem}

Forman also discretized the basic homotopy theorem of Morse theory \cite{milnor16}.

\begin{theorem}[Homotopy]
\label{theorem:forman_homotopy}
Any discrete gradient vector field $(K,M)$ of $S$ is homotopy equivalent to a decomposition $\overline{K}$  whose $p$-cells are the critical $p$-cells of $(K,M)$. 
\end{theorem}	

The basic ingredient in the construction of $\overline{K}$ is Forman's definition 
(\cite{forman98_1}) of  a \textit{(discrete) gradient path} of dimension $p$, which is a sequence  of $p$-cells in $K$,
$\gamma=\sigma_1\sigma_2\sigma_3\ldots\sigma_k,$
such that for each $0\leq i <k$ there is a $(p+1)$-cell $\Sigma$ which contains $\sigma_i$ and $\sigma_{i+1}$ satisfying $\{\sigma_i,\Sigma\}\in M$ and $\{\sigma_i,\Sigma\}\in H(K)$. 
The vertices of the Hasse diagram $H(\overline{K})$ are the critical cells. Given a gradient path  $\sigma_1\sigma_2\sigma_3\ldots\sigma_k$ of dimension $p$, for which $\sigma_1$ belongs to a critical $(p+1)$-cell and $\sigma_k$ is also critical, we introduce an edge of $H(\overline{K})$ joining both critical cells.

Following \cite{gross1987topological}, a {\it rotational system} $\mathcal R$ on a finite, connected graph $G= (V,E)$ consists of a choice of a cyclic ordering on each set of edges with a common vertex. 
A vertex which belongs to a single edge $e$ is associated to the ordering $e \mapsto e$: we suppose there are two copies of $e$. We denote a graph $G$ endowed with a rotational system $\mathcal R$ by $G^\mathcal R$.

We say that $G^\mathcal R$ and ${\tilde G}^{\tilde {\mathcal R}}$ {\it equivalent} if $G$ and $\tilde G$ are isomorphic and the isomorphism either preserves or reverses the cyclic orderings for all vertices. 

Similarly, one may suppose that a decomposition $K = (V,E,F)$ of a connected oriented surface $S$ admits  cyclic orderings of the edges sharing a vertex --- the orientability induces one preferred rotational system in the graph $G=(V,E)$, which we denote by $G^{\mathcal R}$. The next theorem states that we reconstruct $S$ from $G^{\mathcal R}$. 

\begin{theorem} \label{theorem:rotationalsystems}
Every rotational system $\mathcal R$  on a finite graph $G=(V,E) $ induces a  2-cell embedding in some oriented surface $S$ with a decomposition $K=(V,E,F)$.   All such embeddings are equivalent.
\end{theorem}

 In a nutshell, discrete and continuous gradient vector fields are related by the following operations. From a discrete gradient vector field $(K,M)$ of an oriented surface $S$, we obtain the simpler field $
(\overline{K},\emptyset)$ by Theorem \ref{theorem:homotopy}, and its Hasse diagram is endowed of the natural rotational system induced by $S$, giving rise to $H(\overline{K})^{\mathcal{R}}$. One can construct $H(\overline{K})$ directly from $(K,M)$ by considering the gradient paths $\gamma$ between critical points defined above --- they play the role of separatrices in the discrete context. On the other hand, following \cite{peixoto73}, the equivalence class under topological equivalence $[(S,X)]$ of a gradient vector field $(S,X)$ corresponds to the so called Peixoto orgraph $\chi^{\mathcal{R}}$  in \cite{nikolaev97}. 
\[(K,M)\ \ \xrightarrow{\mbox{Homotopy  }}\ \ {(\overline{K},\emptyset)}\ \  \xleftrightarrow{\mbox{Hasse}}\ \ H(\overline{K})^{\mathcal{R}}\ \ \xleftrightarrow{\iota}\ \ \chi^{\mathcal{R}} \ \ \xleftrightarrow{\mbox{Peixoto \cite{peixoto73}}}\ \ [(S,X)]\]
It is not hard to identify  $H(\overline{K})^{\mathcal{R}}$ and $\chi^{\mathcal{R}}$. If these graphs have only two vertices, $S$ is a sphere and the identification is trivial: we handle the more general situation. Both graphs are tripartite: the Hasse diagram by cell dimension and the Peixoto orgraph, by minima, saddle points and maxima of $X$. Edges in the $H(\overline{K})$ --- the 1-cells of $\overline K$ --- connect to two 0-cells and two 2-cells. The same happens for the Peixoto orgraph: saddles connect to two minima and two maxima. Again, both graphs induce 2-cell embeddings in $S$ with 4 edges in the boundary walk of each face. Since $S$ is oriented, both rotational systems are compatible.

%=======================================================================
\section{Discrete gradient line fields}
\label{section:discrete_line_fields}
%=======================================================================
{ 
Let $K = (V, E, F)$ be a decomposition of a compact surface $S$ and $M_V$ be a Morse matching on the subgraph $H_V$ of $H(K)$. The pair $(K,M_V)$ is {\it discrete (restricted) gradient line field}. The relationship between this concept and the appropriate subclass of  continuous line fields is not obvious at this point.  Here we prove the expected basic results.
\begin{figure}[h]
\begin{center}
\includegraphics[width=14cm]{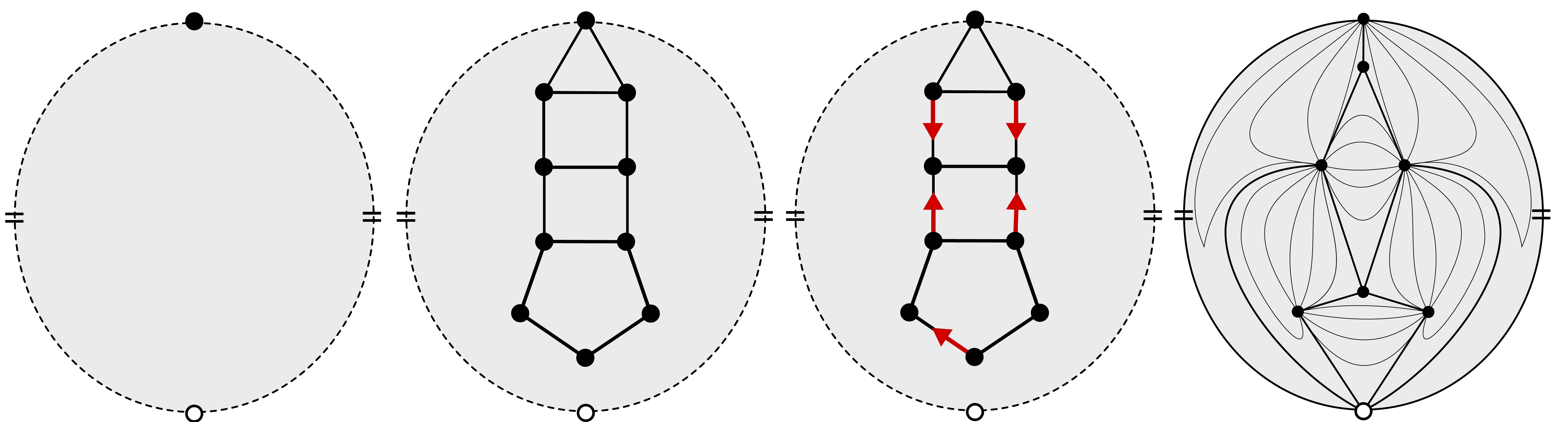}
\caption{A sphere, a decomposition, discrete and continuous line fields. }
\label{figure:sphere_complex_linefield}
\end{center}
\end{figure} 

A vertex $v \in (K, M_V) $ is {\it critical} if it is  unmatched in $M_V$. Its {\it index} is $1$ if it is critical and $0$ otherwise. Let $C(\Sigma)$ be the number of unmatched edges in the boundary walk of a face $\Sigma$. Then  $\Sigma$ is {\it critical}  if $C(\Sigma)\neq 2$. Its index is $1-C(\Sigma)/2$. As an example, let $K=(V,E,F)$ and $T$ be a spanning tree on the graph $(V,E)$. Then $T$ induces a Morse matching  $M_V$ in $(V,E)$, hence a discrete gradient line field $(K, M_V)$ with a single critical cell, the root of $T$. In Figure \ref{figure:sphere_complex_linefield}, the pentagon is a critical face. 

\begin{theorem}[Euler-Poincaré]
\label{theorem:euler_formula}
Let $(K,M_V)$ be a discrete gradient line field on $S$. Then
$$\chi(S)=\displaystyle \sum_{v\in V}index(v)+\sum_{\Sigma\in F}index(\Sigma)$$
\end{theorem}
\begin{theorem}[Homotopy]
\label{theorem:homotopy}
The decomposition $K$ of a discrete gradient line field $(K,M_V)$ is homotopy equivalent to a decomposition $\overline{K}$ of a line field $(\overline{K},\emptyset)$ whose $p$-cells, for $p=0,2$, are the critical $p$-cells of $K$. The indices of the cells are preserved.
\begin{figure}[!h]
\begin{center}
\includegraphics[width=7cm]{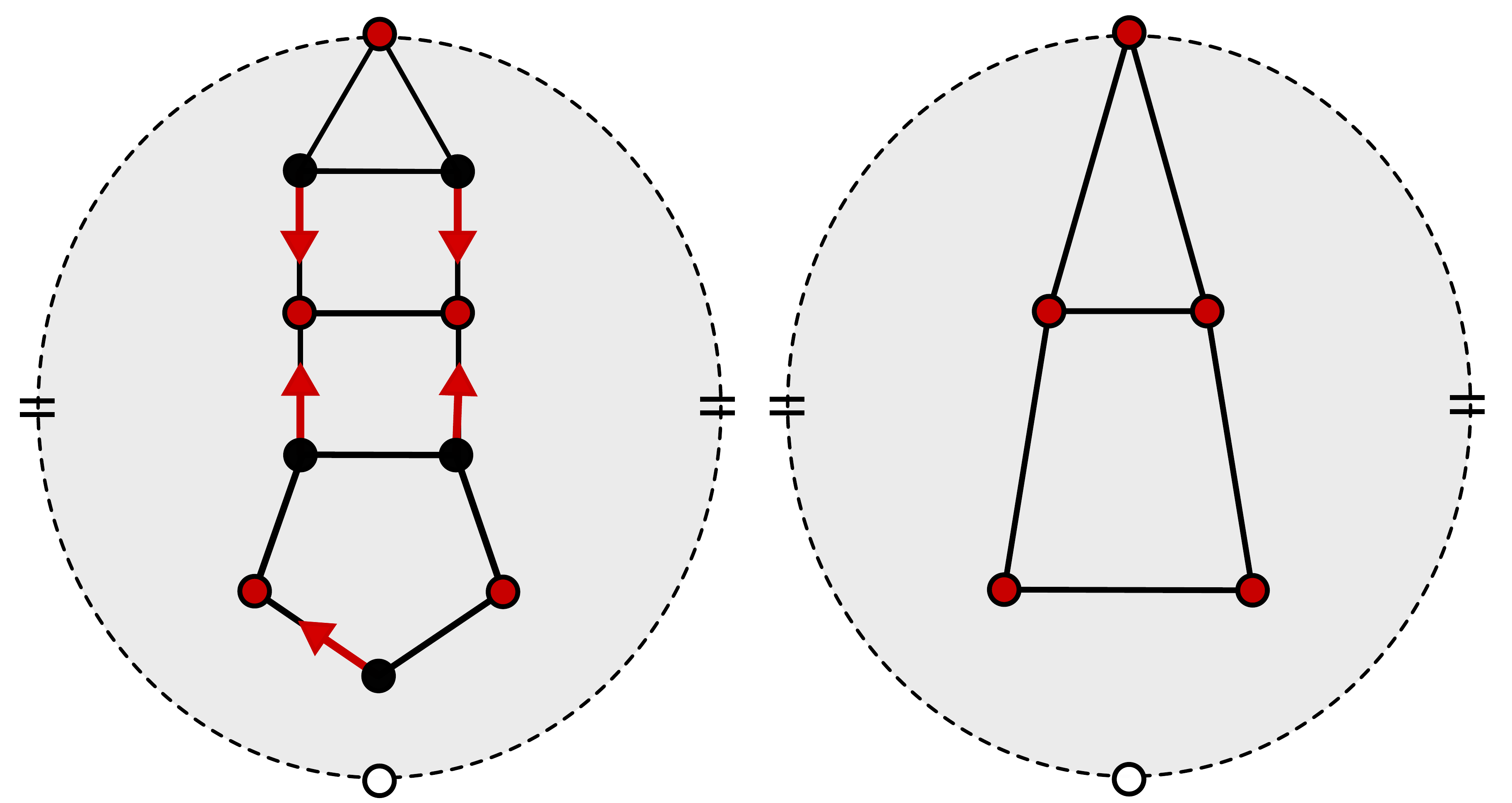}
\caption{$K$ and $\overline{K}$, two homotopy equivalent decompositions. }
\label{figure:last_theorem}
\end{center}
\end{figure} 
\end{theorem}
 
%The two decompositions in Figure \ref{figure:last_theorem} provide an example of the last theorem.

Informally, our main result states that discrete restricted gradient line fields $(K,M_V)$ codify equivalence classes $[(S,L)]$ of (restricted) acyclic line fields which do not admit the last two types of critical points listed in Figure \ref{figure:critical_point}. 
We follow the line of thought presented at the end of the previous section. The first arrow in the diagram  is Theorem \ref{theorem:homotopy}. The Hasse diagram has to be replaced by a {\it radial graph}, to be defined below. 
$$(K,M_V)\ \ \xrightarrow{\mbox{Homotopy }}\ \ {(\overline{K},\emptyset)}\ \  \xleftrightarrow{\mbox{radial graph}}\ \ R(\overline{K})^{\mathcal{R}}$$
In analogy to the equivalence classes defined in the continuous contexts, we  identify two discrete line fields $(K,M_V)$ and $(K',\widetilde{M_V})$ with isomorphic radial graphs $R({\overline{K}})^\mathcal{R}$ and
$R(\overline{K'})^\mathcal{R}$.

In a tradition dating back to the construction of Riemann surfaces, Bronshtein and Nikolaev \cite{nikolaev97} approached the study of a Morse-Smale line field $(S,L)$  by considering a double covering $\widetilde S$ of $S$ with branches at the non-orientable critical points (the last four cases in Figure \ref{figure:critical_point}) on which the field $L$ suspends to a {\it vector} field $(\widetilde S,\widetilde L)$. Clearly, a vector field obtained in such a way gives rise to a Peixoto orgraph ${\tilde \chi}^{\mathcal{R}}$ with an additional symmetry, incorporated in their definition as an involution. Equivalence classes of gradient line fields are in correspondence with  Peixoto orgraphs \cite{nikolaev97}. We might follow their approach, by defining double coverings of discrete line fields $(K,M_V)$. Instead, taking into account the possible implementation of the constructions we suggest in this paper, we try to avoid the doubling of cells: we  project  ${\tilde \chi}^{\mathcal{R}}$ to obtain a {\it folded orgraph} $ \chi^{\mathcal{R}}$, which we define and show to be a radial graph in the section dedicated to proofs.

$$R(\overline{K})^{\mathcal{R}}  \hookleftarrow \chi^\mathcal{R}\ \ \xleftarrow{\mbox{Projection }} \ \ \tilde{\chi}^{\mathcal{R}} \xleftrightarrow{\mbox{BN \cite{nikolaev97} }}\ \ [(S,L)] $$

% \[(K,M_V)\ \ \xrightarrow{\mbox{Homotopy }}\ \ {(\overline{K},\emptyset)}\ \  \xleftrightarrow{\mbox{radial graph}}\ \ R(\overline{K})^{\mathcal{R}}\ \ \xrightarrow{\iota}\ \ \chi^{\mathcal{R}} \ \ \xleftrightarrow{\mbox{BN \cite{nikolaev97} }}\ \ [(S,L)]\]

}
% %=======================================================================
% \section{Integration in discrete gradient line fields }
% \label{subsection:integration}
% %====================================================================== 
 As in \cite{archdeacon1992medial}, we define the \textit{radial graph} $R(K)$ of a decomposition $K=(V,E,F)$ of $S$. Its vertices consist of the vertices of $K$ and a point for each face of $K$. The edges indicate the adjacency relations between vertices and faces, and are represented by disjoint arcs in $S$. Thus $R(K)$ is embedded in $S$. The next result is a characterization of radial graphs \cite{archdeacon1992medial}.

\begin{theorem}
\label{theorem:radial_graph}
Embed a bipartite graph $G$ on a surface $S$  yielding a decomposition $K$. Suppose that either all  2-cells of $K$ have four edges along their boundary walk, or $K$ contains a single 2-cell, and the corresponding boundary walk traverses a single edge twice (in this case, $S$ is a 2-sphere).
Then, up to equivalence, there are only two decompositions $K_1$ and $K_2$ of $S$ whose radial graphs $R(K_i)$ are isomorphic to 
$G$. 
\end{theorem}

The decompositions in the statement are obtained as follows --- we handle the case when $S$ is not the 2-sphere. For $G=(V,E)$, split $V = V_1 \sqcup V_2$ using that $V$ is bipartite. We construct $K_1$. Its set of vertices is $V_1$. Edges are diagonals of each 2-cell joining two vertices in $V_1$. Finally, faces are components of $S \setminus {V_1\cup E}$, and are  identified with the vertices in $V_2$.

\bigskip
As in the previous section, where we constructed  the Hasse diagram of $\overline{K}$ in terms of $K$, we may construct the radial graph $R(\overline{K})$ of $\overline{K}$ directly from $(K,M_V)$ by using the gradient paths of dimension $0$. The vertices of $R(\overline{K})$ are the critical cells (vertices and faces). Each edge corresponds to a pair consisting of a critical face $\Sigma$ and a critical vertex $u$ such that there is a gradient path connecting $u$ with some vertex $v$ in $\Sigma$. Finally, $R(\overline{K})^{\mathcal{R}}$ is obtained from $R(\overline{K})$ by introducing the natural rotational system induced by $S$. From Theorem  \ref{theorem:radial_graph}, the map ${(\overline{K},\emptyset)}\ \longleftrightarrow\ R(\overline{K})^{\mathcal{R}}$ is a bijection. The first   picture in  Figure \ref{figure:sphere_complex_linefield_Ggraph_morsesmale} is a line field $(K, M_V)$ for which we indicated a gradient path, or equivalently, an edge of $R(\overline{K})$. In the second picture, all such edges are represented and the rotational system gives rise to 2-cells. Finally the third picture shows the associated radial graph.

\begin{theorem}
\label{theorem:quads_theorem}
Every equivalence class of  restricted acyclic line fields admits a folded orgraph isomorphic to a radial graph, which in turn corresponds to an equivalence class of  discrete line fields.
\end{theorem}

% \begin{corollary}
% Let $K$ be a decomposition of the compact, oriented surface $S$ and $(K, M_V)$ be a discrete gradient line field. Then from the Peixoto orgraph $X^{\mathcal{R}}$ given rise a continuous line field $\mathcal{L}$ on $S$.
% \end{corollary}

\begin{figure}[h]
\begin{center}
\includegraphics[width=14cm]{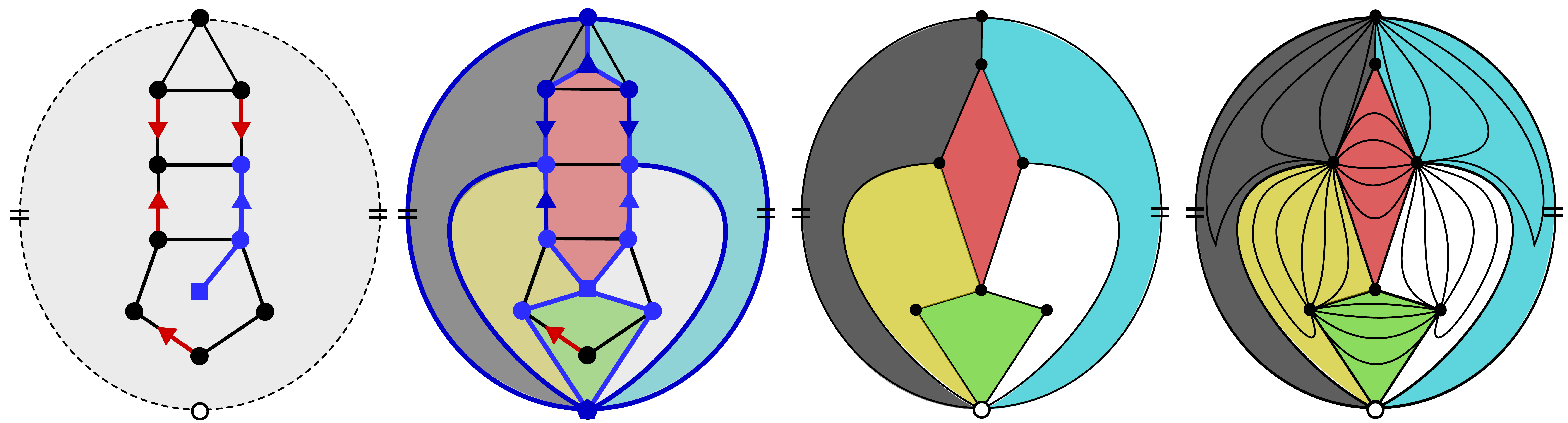}
\caption{Discrete gradient line field $(K,M_V)$. Radial graph $R({\overline K})^{\mathcal R}$. Folded orgraph. Line field.}
\label{figure:sphere_complex_linefield_Ggraph_morsesmale}
\end{center}
\end{figure} 

%=======================================================================
%\appendix
\section{Proofs}
%=======================================================================
This final section is dedicated to sketches of proofs of the main theorems  concerning discrete line fields: the homotopy Theorem \ref{theorem:homotopy}, from which we derive the Euler-Poincaré formula, Theorem \ref{theorem:euler_formula},  and  Theorem \ref{theorem:quads_theorem}.

We use an auxiliary construction --- a Morse matching $M$ on the  Hasse diagram $H(K)$ as defined above Theorem \ref{theorem:forman_euler_formula} --- which enables us to convert the proof of Theorem \ref{theorem:homotopy} into Forman's argument yielding Theorem \ref{theorem:forman_homotopy}. 

Let $(K, M_V)$ be a discrete gradient line field on a compact, oriented surface $S$. Define a graph $G$ as follows: the vertices are the unmatched edges in $M_V$, and the edges are the non critical faces (i.e., faces with exactly two unmatched edges in its boundary walk). As an example, consider Figure \ref{figure:graph_G}.

\begin{figure}[h]
\begin{center}
\includegraphics[width=12cm]{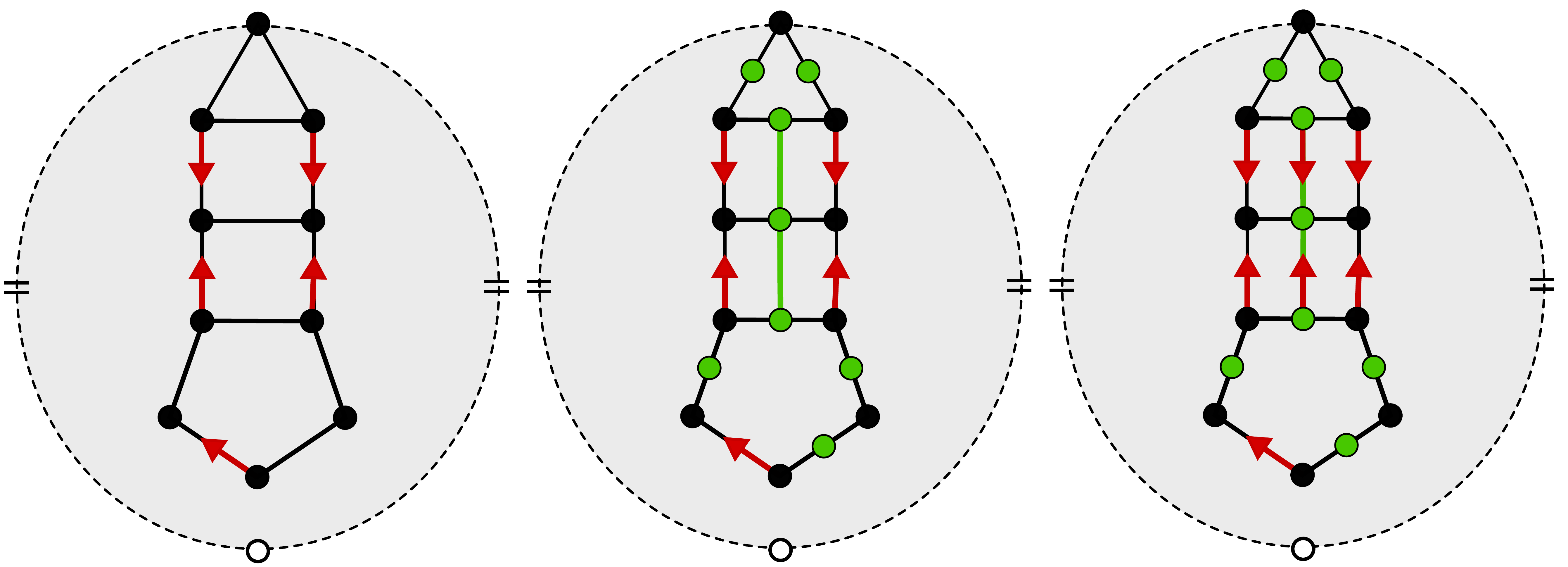}
\caption{Critical cells of $(S,L)$. The graph $G$. The induced Morse matching.}
\label{figure:graph_G}
\end{center}
\end{figure} 	

For each connected component $G_i$ of $G$,  a spanning tree $T_i$   induces a Morse matching $M_i$ in $H_F$ restricted to the cells in $G_i$, as described before Theorem \ref{theorem:euler_formula}. The Morse matching $M_F$ is just the union of the $M_i$'s. Finally the \textit{induced} Morse matching $M$ on $H(K)$ is defined by $M = M_V\cup M _F$.
We treat two annoying cases separately. 

\begin{lemma}
\label{lemma:sphere}
The decomposition induced by a critical face $\Sigma$ of index $1$ is sphere.
\begin{proof}
If the index of a critical face $\Sigma$ is $1$ (when $C(\Sigma)=0$), all edges in the boundary walk of $\Sigma$ are matched with vertices in $\Sigma$, and the matching has to be acyclic, because $M_V$ is acyclic. In other words, the complex induced by the edges in the boundary of $\Sigma$ is a tree. Applying  Theorem \ref{theorem:forman_homotopy} in the complex induced by $\Sigma$ and the matching $M_V$ is a critical vertex and a critical face, which is a sphere. Hence the next lemma follows.
\end{proof}
\end{lemma}

\begin{lemma}
\label{lemma: sphere_projective}
Let $G_i$ be a connected component of $G \subset S$ which contains a cycle $\gamma$, a collection of faces and edges. Then the closure of $\gamma$ is a sphere $\mathbb{S}^2$. 
\begin{figure}[h]
\begin{center}
	\includegraphics[width=14cm]{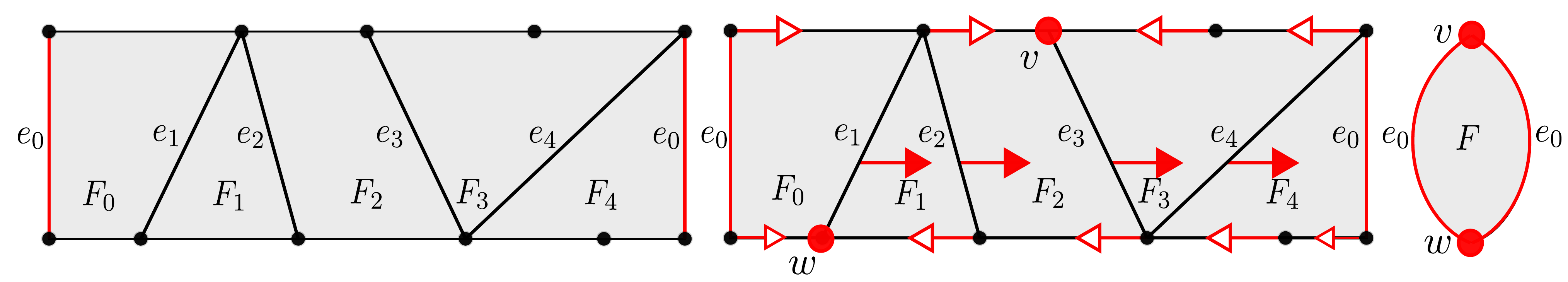}
	\caption{A cycle $\gamma$, the Morse matching induced by $(K, M_V)$ and what's left after simplifying as in Forman's homotopy theorem. }
	\label{figure:sphere_projective}
\end{center}
\end{figure}
\begin{proof}
Take a cycle $\gamma\longleftrightarrow e_0F_0e_1F_1\ldots e_kF_ke_0$ in the graph $G$. 
The matching $M_F$  above has exactly one critical face and one edge in $\gamma$ (see Figure \ref{figure:sphere_projective}), and the rest of edges on the boundary walk of the faces in $\gamma$ are matched, since all face $F_i\in \gamma$ has just two unmatched edges and they belong to $\gamma$. Each edge in $K$ belongs to exactly two faces in $K$. Applying Theorem \ref{theorem:forman_homotopy} in the complex induced by $\gamma$, we are left with the decomposition of a sphere. In particular, all faces of the decomposition $K$ of $S$ belong to the cycle $\gamma$. and $S$ is a sphere or a projective plane. Since $S$ is orientable, the latter case does not happen.
\end{proof}
\end{lemma}

\bigskip
% \paragraph{Proof of Theorem \ref{theorem:homotopy}}
% \begin{proof}
\begin{proof}[ Proof of Theorem \ref{theorem:homotopy}]
By  Theorem \ref{theorem:forman_homotopy} applied to the induced Morse matching $M$ in $H(K)$, we learn that $K$ is homotopy equivalent to another decomposition $\overline{K}$ with exactly the same critical cells that $M$. To complete the proof,  we must prove that $(K,M_V)$ and $(\overline{K},\emptyset)$ have  the same critical cells and indices.
The cases  handled in the two  lemmas above imply trivial decompositions of $S = \mathbb{S}^2$.

The critical vertices of $M$ and $M_V$ are the same: by definition $M=M_V\sqcup M_F$ and $M_F$ is a matching in $H_F$.

Let $\Sigma$ be a critical face of $(K,M_V)$. The case $C(\Sigma)=1$ is treated in Lemma 8 and $C(\Sigma)=2$ is a noncritical face.   Suppose then $C(\Sigma)> 2$. By construction, $\Sigma$ is not contained on the graph $G$. By the definition of the Morse matching $M$,  $\Sigma$ must be unmatched and hence corresponds to a face $\overline \Sigma$ of $H(\overline K)$. Also, $index(\Sigma) = 1 - C(\Sigma)/2$, where we label the unmatched edges on the boundary of $\Sigma$ as $\{\sigma_1,\sigma_2,\ldots,\sigma_{C(\Sigma)}\}$. Each edge $\sigma_i$ belongs to a connected component $G_i$ of $G$. For each $G_i$ there is exactly one critical edge $\sigma$ and exactly one gradient path from $\sigma_i$ to $\sigma$. Thus $index(\Sigma) = index(\overline{\Sigma})$, since the adjacency of faces in $\overline{\Sigma}$ are given by the gradient paths of $M$  from  edges in the boundary of $\Sigma$ to critical edges.
\end{proof}
% Suppose all faces are regular ($C(\Sigma)=2$ for all face $\sigma$), so the graph $G$ must be a cycle. Using the Lemma \ref{lemma: sphere_projective} and that $K$ is a decomposition of an oriented compact surface then it is a sphere and $M$ says that $\overline{K}$ is homotopy equivalent two vertices, one edge and one face. In the discrete gradient line field language we have that $\overline{K}$ has also no critical face. So the proof is completed.
%\end{proof}

\begin{proof}[ Proof of Theorem \ref{theorem:euler_formula}]
From Theorem \ref{theorem:homotopy}, it is enough to prove the case $M_V=\emptyset$. By  Euler's formula,
$$\chi(K)=\sharp V-\sharp E+\sharp F.$$
Since all vertices are critical, $\sharp V=\displaystyle \sum_{v\in V}index(v)$. In the sum
$$
\displaystyle \sum_{\Sigma\in F}index(\Sigma) = \sharp F-\sum_{\Sigma\in F} \frac{C(\Sigma)}{2},
$$ 
each edge $e$ is contained in the boundary walk of exactly two faces in $K$ and the sum on the right hand side then 
is the number of edges in  $K$. 
\end{proof}	

We now prepare for the proof of Theorem \ref{theorem:quads_theorem}.
Following \cite{nikolaev97}, a 2-branched covering $p:\widetilde S\to S$ induces the suspension of a Morse-Smale line field $(S,L)$ to a vector field $(\widetilde S, V = \widetilde L)$  with branches at the non-orientable critical points of $L$. The vector field in turn gives rise to a \textit{Peixoto orgraph} ${\tilde \chi}^{\mathcal{R}}$, an orgraph ${\tilde \chi}$ with a rotational system $R$. More precisely, 
the orgraph $\widetilde{\chi}$ is a tripartite graph, consisting of an upper level $Max$  of maximum (critical) points, a lower level $Min$ of minimum points and a middle level $Sad$ containing the saddles with  2, 4 or 6 separatrices. Additional properties follow from the fact that $V$ was obtained from a double covering. Thus,  there is an involution $\theta$ in the set of vertices which induces a bijection  $\theta:Max \to Min$. Also, $\theta$ restricts to $Sad$, keeping fixed the saddles with 2 and 6 separatrices and coupling saddles with 4 separatrices. The edge set of  $\widetilde{\chi}$ consist of the separatrices joining saddles to vertices in the extremal levels.

The orientation of the orbits in $\widetilde L$ induces naturally a orientation in the edges of $\widetilde{\chi}$. If $Sad=\emptyset$ then $\widetilde{\chi}$ consists of a pair maximum-minimum  connected by an edge --- $\widetilde S$ is a sphere. Finally the rotational system $\mathcal{R}$ is given by the orientation of the surface $\widetilde S$ and induces, as in Theorem \ref{theorem:rotationalsystems}, a collection of faces.

\begin{lemma}
Every face obtained from the Peixoto orgraph ${\tilde \chi}^{\mathcal{R}}$ of an acyclic Morse-Smale line field $(S,L)$ with $Sad \ne \emptyset$ has 4 edges in its boundary walk. 
\begin{proof}
Let $\Sigma$ be a face of ${\tilde \chi}^{\mathcal{R}}$. Clearly, the are no separatrices between saddles (since the fields are Morse-Smale) and between maxima and minimum. In particular, the number of edges in the boundary walk of $\Sigma$ is even. Now attach two copies of $\Sigma$ along its boundary, which becomes an equator to a sphere $(\mathbb{S}^2, V)$ with copies of the field in $\Sigma$ on each hemisphere.
In the sphere, the vertices of $\Sigma$ which represent saddle points of $\widetilde{L}$ have index zero (they are destroyed by small perturbations) and  maxima and minima are maintained. Since the Euler characteristic of $(\mathbb{S}^2, V)$ is 2, the total number of maxima and minima must be 2.  
\end{proof}
\end{lemma}

We define the \textit{folded Peixoto orgraph} $\chi^\mathcal{R}$ of a Morse-Smale line field $(S,L)$ as the graph combined with a rotational system obtained by projecting the Peixoto orgraph $\widetilde{\chi}^\mathcal{R}$ with the 2-branched covering $p:\widetilde S\to S$.

\begin{proof}[Proof of Theorem \ref{theorem:quads_theorem}] Consider a Peixoto orgraph ${\tilde \chi}^{\mathcal{R}}$ with $Sad \ne \emptyset$.
We must prove that the folded Peixoto orgraph $\chi^\mathcal{R}$ is a radial graph, that is, a bipartite graph with a rotational system giving rise to faces with boundary walks with four edges.

The involution $\theta$ on the vertices of ${\tilde \chi}^{\mathcal{R}}$
leads a partition of the vertices of $\chi$ on two levels, $p(Max) = p(Min)$ and the projection $p(Sad)$.

A face $\widetilde \Sigma \subset \widetilde S$ obtained from the Peixoto orgraph ${\tilde \chi}^{\mathcal{R}}$ is a disk with four edges in its boundary walk, by the previous lemma. The map $p:\widetilde S\to S$ restricted to the interior of $\widetilde \Sigma$ is a homeomorphism since all the branched points are concentrated in the vertices of ${\tilde \chi}$. Thus the projection $\Sigma = p(\widetilde \Sigma)$ in $\chi^\mathcal{R}$ is a face which  4 edges in its boundary walk. 
\end{proof}

Notice that the presence of the non-stable critical points is irrelevant for the argument.

% %=======================================================================
% \section{Conclusion and Future Work}
% \label{section:conclusion}
% %=======================================================================

% General construction of discrete gradient line fields and discrete Morse-Smale foliations.

% Self-contained definition of the general case, which entails Forman's constructions, i.e. for the particular case when the line field is indeed a gradient field.

% Very simple algorithmic constructions:
% \begin{itemize}
% \item[Singularity detection] use Definition xxx, linear complexity.
% \item[Separatrix integration] direct propagation, either always splits or always merge, thus linear complexity.
% \item[Morse-Smale foliation] determine for each element, which Morse-Smale cell it belongs to. Essentially an integration, like previous case, thus a linear complexity.
% \end{itemize}
% The simplicity and combinatorial nature of the above constructions motivate their use in geometry processing applications. The main challenge in that context is to match the construction of a discrete gradient line field with the geometry of the surace, e.g. a max curvature line field.

% Also generalize to higher dimension...

%=======================================================================
\bibliography{lipics-v2016-sample-article}
%=======================================================================

\end{document}